\documentclass[12pt]{amsart}

\usepackage{amssymb}
\usepackage{setspace}
\usepackage[usenames, dvipsnames, svgnames]{xcolor}

\usepackage{url}

\usepackage[colorlinks=true,linkcolor=blue,urlcolor=blue,citecolor=blue, pagebackref]{hyperref}

\usepackage{enumerate} % for using different counters
\usepackage{helvet} % for Sans Serif fonts
\usepackage{courier} % default ttdefault
\usepackage{eucal} % better calligraphic fonts
\usepackage{mathrsfs} % for fancy calligraphic fonts

\reversemarginpar %marginal notes appear on the left side margin
\normalmarginpar %switches back to right side margin

\let\oldmarginpar\marginpar %changes the font of margin text
\renewcommand\marginpar[1]{\-\oldmarginpar{\raggedright\small\sf #1}}

\newcommand{\nc}{\newcommand}

\nc{\rnc}{\renewcommand}

\nc{\bs}{\backslash}
\nc{\te}{\otimes}
\nc{\lf}{\lfloor} %for round down
\nc{\rf}{\rfloor}
\nc{\lc}{\lceil}  %for round up
\nc{\rc}{\rceil}
\nc{\lr}{\longrightarrow}
\nc{\sr}{\stackrel}
\nc{\dar}{\dashrightarrow}
\nc{\thra}{\twoheadrightarrow}

\nc{\la}{\langle}
\nc{\ra}{\rangle} 

\nc{\ms}{\mathscr}
\nc{\mc}{\mathcal}
\nc{\mb}{\mathbb}
\nc{\mf}{\mathbf}
\nc{\mr}{\mathrm}
\nc{\mg}{\mathfrak}

\nc{\bP}{\mathbb{P}}
\rnc{\P}{\mathbb{P}}
\nc{\Q}{\mathbb{Q}}
\nc{\Z}{\mathbb{Z}}
\nc{\C}{\mathbb{C}}
\nc{\R}{\mathbb{R}}
\nc{\A}{\mathbb{A}}
\nc{\V}{\mathbb{V}}
\nc{\W}{\mathbb{W}}
\nc{\N}{\mathbb{N}}
\nc{\D}{\mathbb{D}}
\nc{\G}{\mathbb{G}}
\nc{\F}{\mathbb{F}}
\nc{\qb}{\overline{\mathbb{Q}}}

\nc{\del}{\partial}

\nc{\wt}{\widetilde}
\nc{\wh}{\widehat}
\nc{\ov}{\overline}
\nc{\un}{\underline}

\nc{\naive}{\!\sim_n}

\nc{\Spec}{\mr{Spec}}

\nc{\omx}{\omega_X}

\nc{\ep}{\epsilon}
\nc{\ve}{\varepsilon}
\nc{\vt}{\vartheta}

\rnc{\l}{\lambda}
\rnc{\k}{\kappa}

\nc{\ovl}{\ov{\lambda}}
\nc{\vl}{\mb{V}_{\ovl}}
\nc{\dl}{\mb{D}_{\ovl}}
\nc{\mnb}{\ov{\mr{M}}_{0,n}}
\nc{\mn}{\mr{M}_{0,n}}
\nc{\mel}{\ov{\mr{M}}_{1,1}}
\nc{\mfb}{\ov{\mr{M}}_{0,4}}
\nc{\mof}{\mr{M}_{0,4}}
\nc{\mgnb}{\ov{\mr{M}}_{g,n}}
\nc{\mgn}{\ov{\mr{M}}_{g,n}}
\nc{\omc}{\ov{\mr{M}}}

\rnc{\sl}{\shoveleft}

\nc{\res}{\operatorname{Res}}
\nc{\pic}{\operatorname{Pic}}
\nc{\spec}{\operatorname{Spec}}
\nc{\im}{\operatorname{Im}}
\nc{\gal}{\operatorname{Gal}}
\nc{\fr}{\operatorname{Fr}}
\nc{\ed}{\operatorname{ed}}
\nc{\rank}{\operatorname{rank}}
\nc{\h}{\operatorname{H}}
\nc{\ch}{\operatorname{char}}
\nc{\sw}{\operatorname{sw}}
\nc{\rsw}{\operatorname{rsw}}
\nc{\supp}{\operatorname{supp}}

\nc{\Mor}{\operatorname{Mor}}
\nc{\Per}{\operatorname{Per}}
\nc{\prep}{\operatorname{Prep}}
\nc{\End}{\operatorname{End}}
\nc{\Orb}{\operatorname{Orb}}

\nc{\tr}{\operatorname{Tr}}

\nc{\br}{\bar{\rho}}

\newtheorem{thm}{Theorem}[section]
\newtheorem{prop}[thm]{Proposition}

\newtheorem{lem}[thm]{Lemma}

\theoremstyle{definition}
\newtheorem{defn}[thm]{Definition}

\newtheorem{rem}[thm]{Remark}

\numberwithin{equation}{section}

\input{xypic}

\title{Lifting $G$-irreducible but $\mr{GL}_n$-reducible Galois
  representations} \thanks{We are grateful to Wushi Goldring for
  stimulating conversations.  N.F. was supported by the DAE,
  Government of India, project no. RTI4001. C.K.  would like to thank
  TIFR, Mumbai for its hospitality, in periods when some of the work
  was carried out. S.P. was supported by NSF grants DMS-1700759 and
  DMS-1752313. We also thank the referees for their comments and
  corrections which helped to improve the exposition.  }

\begin{document}

\author[N.~Fakhruddin]{Najmuddin Fakhruddin}
\address{School of Mathematics, Tata Institute of Fundamental Research, Homi Bhabha Road, Mumbai 400005, INDIA}
\email{naf@math.tifr.res.in}
\author[C.~Khare]{Chandrashekhar Khare}
\address{UCLA Department of Mathematics, Box 951555, Los Angeles, CA 90095, USA}
\email{shekhar@math.ucla.edu}
\author[S.~Patrikis]{Stefan Patrikis}
\address{Department of Mathematics, The University of Utah, 155 S 1400 E, Salt Lake City, UT 84112, USA}
\email{patrikis@math.utah.edu}

%\doublespacing

\begin{abstract}
In recent work, the authors proved a general result on lifting $G$-irreducible odd Galois representations $\gal(\overline{F}/F) \to G(\ov{\F}_{\ell})$, with $F$ a totally real number field and $G$ a reductive group, to geometric $\ell$-adic representations. In this note we take $G$ to be a classical group and construct many examples of $G$-irreducible representations to which these new lifting methods apply, but to which the lifting methods currently provided by potential automorphy theorems do not.
\end{abstract}
\maketitle

\section{Introduction}
Let $G$ be a smooth group scheme over $\Z_{\ell}$ such that $G^0$ is a
split connected reductive group scheme, and $G/G^0$ is finite of order
prime to $\ell$. Let $F$ be a number field with algebraic closure
$\ov{F}$ and absolute Galois group $\Gamma_F= \gal(\ov{F}/F)$, and let
$\br \colon \Gamma_F \to G(\ov{\F}_{\ell})$ be a continuous
homomorphism. The question of whether $\br$ admits a geometric (or,
more ambitiously, automorphic or motivic) lift $\rho \colon \Gamma_F
\to G(\ov{\Z}_\ell)$ has attracted a great deal of interest at least
since Serre formulated his modularity conjecture in the case $F=\Q$,
$G= \mr{GL}_2$, and $\br$ absolutely irreducible and odd (see
Definition \ref{defn:odd} below). There are essentially two types of
methods, first studied for $G= \mr{GL}_2$ and $F$ totally real, for
proving such lifting theorems, a purely Galois-theoretic approach
developed by Ramakrishna (\cite{ramakrishna:lifting},
\cite{ramakrishna02}), and an approach developed in
\cite{khare-wintenberger:serre0} and \cite{khare:serrelevel1} making
crucial use of potential automorphy (\cite{taylor:remarks}). Much
further work on both of these methods (and their descendants) has led
to the papers \cite{fkp:reldef}, for the Galois-theoretic method, and
\cite{blggt:potaut}, for the automorphic methods, which more or less
represent the state-of-the-art (see also \cite{calegari-emerton-gee}
for a refinement of the local hypotheses in \cite{blggt:potaut}).

In the present paper, we will discuss examples when $G$ is a classical
group, and where the Galois-theoretic methods of \cite{fkp:reldef}
apply, but where current potential automorphy methods do not. The
basic feature here constraining the range of applicability of
potential automorphy methods is the distinction between $\br$ that are
$G$-irreducible and $\br$ that remain irreducible after composition
with a faithful representation $G \to \mr{GL}_N$: recall that $\br$ is
$G$-irreducible if its image is contained in no proper parabolic
subgroup of $G(\ov{\F}_{\ell})$.

We first carry out our analysis for examples (\S \ref{elementary} and \S \ref{pIGP}) whose construction is rather straightforward, but which we hope serves to convince the reader of the ubiquity of the phenomenon we are highlighting; in \S \ref{elementary} we simply sum two-dimensional representations, while in \S \ref{pIGP} we explain a quite general but ``soft" approach to constructing examples using Moret-Bailly's result on the ``potential inverse Galois problem with local conditions" (\cite{moret-bailly}). The bulk of the paper treats a more elaborate example (\S \ref{zywinasection}-\S \ref{complexconj}) arising from orthogonal representations $\Gamma_{\Q} \to \mr{SO}_N(\mathbb{F}_{\ell})$ constructed in Zywina's work on the inverse Galois problem (\cite{zywina:orthogonal}). 

Whether one approaches lifting problems via potential automorphy or just via Galois deformation theory, a critical hypothesis is the following generalization of Serre's ``oddness" condition:
\begin{defn}\label{defn:odd}
We say $\br \colon \Gamma_F \to G(\ov{\F}_{\ell})$ is odd if for all $v \mid \infty$,
\[
h^0(\Gamma_{F_v}, \br(\mg{g}^{\mr{der}})) = \dim (\mr{Flag}_{G^0}),
\]
where $\br(\mg{g}^{\mr{der}})$ is the Lie algebra of the derived group $G^{\mr{der}}$ of $G^0$, equipped with the action of $\Gamma_F$ via the composite $\mr{Ad}\circ \br$, and $\mr{Flag}_{G^0}$ is the flag variety of $G^0$.
\end{defn}
In the examples of \S \ref{elementary} and \S \ref{pIGP}, this oddness condition will follow immediately from the construction. Most of our technical work is to establish the oddness condition in Zywina's examples; we deduce it from a topological calculation of the action of complex conjugation on (a piece of) the cohomology of a real elliptic surface. We hope these calculations may be of some independent interest.  

We state  the main theorem of this note (this is Theorem \ref{thm:main} below).

\begin{thm} \label{thm:main-intro}%
Let $N \geq 6$ be an even integer, and let $\ell \gg_N 0$ be a sufficiently large prime. 
\begin{enumerate}
\item There is a totally real field $F$ which is solvable over $\Q$
  and infinitely many non-isomorphic Galois representations
\[
\br \colon \Gamma_F \to \mr{SO}_{N+1}(\ov{\F}_{\ell})
\]
such that $\br$ is irreducible as an $\mr{SO}_{N+1}$-valued representation, but reducible as a $\mr{GL}_{N+1}$-valued representation, and $\br$ admits a geometric lift $\rho \colon \Gamma_F \to \mr{SO}_{N+1}(\ov{\Z}_{\ell})$ with Zariski-dense image.
\item Moreover, if we assume that representations
  $\Gamma_{\Q_{\ell}} \to \mr{SO}_{N+1}(\ov{\F}_{\ell})$ admit Hodge--Tate
  regular de Rham lifts
  $\Gamma_{\Q_{\ell}} \to \mr{SO}_{N+1}(\ov{\Z}_{\ell})$, then in part
  (1) we may take $F=\Q$.
\end{enumerate}
\end{thm}

Let us explain the restriction in this theorem that gives us only a conditional result when $F= \Q$. In addition to oddness, the other requirement for producing lifts (via either of the methods described above) of residual representations $\br \colon \Gamma_{F} \to G(\ov{\mathbb{F}}_{\ell})$ is that $\br|_{\Gamma_{F_v}}$ should have lifts at each place $v$ of $F$---of course, this is obvious unless $\br$ is ramified at $v$---and that when $v$ lies above $\ell$ these lifts may be taken to be de Rham and Hodge--Tate regular. It is perhaps to be expected that this is no condition at all, but it is not known in this generality, and unfortunately we have nothing to add at present to this question. In our examples, the most serious obstacle is the existence of Hodge--Tate regular lifts when $v \vert \ell$, and we will mostly avoid this by producing examples over some number field $F$ allowing us to rig the local behavior of our $\br$. That said, in contrast to the examples based on the potential inverse Galois problem,  in the theorem above arising from Zywina's examples  we  could  take $F=\Q$,  if we knew that representations $\Gamma_{\Q_{\ell}} \to \mr{SO}_{N+1}(\F_{\ell})$ admit a Hodge--Tate regular de Rham lift. Recently the techniques for producing such lifts have greatly advanced: work of Emerton--Gee (\cite{emerton-gee:moduli}) addresses this question with $\mr{GL}_{N+1}$ in place of $\mr{SO}_{N+1}$, and others are at work extending the methods of \cite{emerton-gee:moduli} to other groups. In the examples arising from Moret-Bailly's theorem (\S \ref{pIGP}), we cannot necessarily ensure that $F/\Q$ is solvable as in Theorem \ref{thm:main-intro}.

We make some further remarks on the relation of our examples to the
automorphic theory. In \S \ref{pIGP} and \S \ref{zywinasection}, the
representations we construct have the form
$\br= \bar{\theta} \oplus 1$, where $N$ is an even integer, for some
$\bar{\theta} \colon \Gamma_F \to \mr{SO}_N(\F_{\ell})$ with ``large"
image, so $\br$ is valued in $\mr{SO}_{N+1}(\F_{\ell})$ and is
$\mr{SO}_{N+1}$-irreducible but $\mr{GL}_{N+1}$-reducible. We note
that our examples suggest, on the automorphic side, the existence of
congruences between certain non-endoscopic cusp forms on
$\mr{Sp}_{N}(\mathbb{A}_F)$ and certain endoscopic forms associated to
its (split) endoscopic group $\mr{SO}_N$. Moreover, when
$N \equiv 2 \pmod 4$, $\bar{\theta}$ itself is not odd and cannot be
lifted using the methods of \cite{fkp:reldef}, \cite{blggt:potaut}, or
indeed of any hoped-for extension of potential automorphy methods.
Indeed, we expect $\bar{\theta}$ will only have geometric
$\mr{SO}_N(\ov{\Z}_{\ell})$-lifts, if any, with some multiplicity
greater than $1$ in its Hodge--Tate weights (as $\mr{GL}_N$-valued
representation), and heuristically one expects its deformation ring
(parametrizing de Rham lifts with fixed Hodge--Tate data) to be
zero-dimensional. In many cases we prove that $\bar{\theta}$ indeed
has no regular lift in Proposition \ref{distinctHT} (see also Remark
\ref{notodd}).

It remains an interesting open problem to develop potential automorphy
methods, and corresponding lifting theorems, in a more group-intrinsic
manner, the difference between $G$-irreducibility and $\mr{GL}_n$
irreducibility explored in this paper providing one motivation for
doing so. To do this would require further understanding of Langlands
functoriality, e.g., between automorphic representations of inner
forms of $\mr{GSpin}_{2n+1}$ and $\mr{GL}_{2n}$, including
compatibility with local Langlands at ramified primes. Moreover,
allowing representations that are $G$-irreducible but
$\mr{GL}_n$-reducible would force the Taylor--Wiles method to grapple
with contributions from endoscopic groups where, interestingly,
automorphic representations on an endoscopic group that are not
discrete series at infinity can intervene (again see Remark
\ref{notodd})

\section{A theorem from \cite{fkp:reldef}}

We recall a weakening of the main theorem of \cite{fkp:reldef} (the
full result yields more precise conclusions about the local
restrictions and the image of the lift), which will be our input for
lifting residual representations $\br$ once we have established some
of their local and global properties:
\begin{thm}[See Theorem A of \cite{fkp:reldef}]\label{fkpthmA}
Let $\ell \gg_G 0$ be a prime. Let $F$ be a totally real field, and let $\br \colon \Gamma_F \to G(\ov{\F}_{\ell})$ be a continuous representation unramified outside a finite set of finite places $S$ containing the places above $\ell$. Let $\widetilde{F}$ denote the smallest extension of $F$ such that $\br(\Gamma_{\widetilde{F}})$ is contained in $G^0(\ov{\F}_{\ell})$, and assume that  $[\widetilde{F}(\zeta_\ell): \widetilde{F}]$ is strictly greater than an integer $a_G$ depending only on the root datum of $G$ (see \cite[Lemma A.6]{fkp:reldef}). Fix a geometric lift $\mu \colon \Gamma_F \to G/G^{\mr{der}}(\ov{\Z}_{\ell})$ of $\bar{\mu}:=\br \pmod{G^{\mr{der}}}$, and assume that $\br$ satisfies the following:
\begin{itemize}
\item $\br$ is odd.
\item $\br|_{\Gamma_{\widetilde{F}(\zeta_\ell)}}$ is absolutely irreducible.
\item For all $v \in S$, $\br|_{\Gamma_{F_v}}$ has a continuous lift
  $\rho_v \colon \Gamma_{F_v} \to G(\ov{\Z}_{\ell})$  such that
  $ \rho_v   \pmod{G^{\mr{der}}} =  \mu|_{\Gamma_{F_v}} $; and that for $v \vert \ell$ this lift may be chosen to be de Rham and regular in the sense that the associated Hodge--Tate cocharacters are regular.
\end{itemize}
Then there is a lift
\[
\xymatrix{
& G(\ov{\Z}_{\ell}) \ar[d] \\
\Gamma_F \ar[r]_-{\br} \ar[ur]^{\rho} & G(\ov{\F}_{\ell}) 
}
\]
of $\br$ satisfying:
\begin{itemize}
\item The projection of $\rho$ to $G/G^{\mr{der}}(\ov{\Z}_{\ell})$ equals  $\mu$.
\item $\rho$ is unramified outside a finite set of primes, and the restrictions $\rho|_{\Gamma_{F_v}}$ for $v \vert \ell$ are de Rham and regular, having the same $\ell$-adic Hodge type as $\rho_v$.
\item The Zariski-closure of the image $\rho(\Gamma_F)$ contains $G^{\mr{der}}$.
\end{itemize}
\end{thm}

\section{An elementary example}\label{elementary}

In this brief section, we consider a relatively simple example where
$G= \mr{GSp}_{2n}$, defined with respect to the symplectic form
$J= \begin{pmatrix} 0 & 1_n \\ -1_n & 0 \end{pmatrix}$, where $1_n$
denotes the $n\times n$ identity matrix. Let
$\br_1, \br_2, \ldots, \br_n \colon \Gamma_\Q \to
\mr{GL}_{2}(\ov{\F}_{\ell})$ be $n$ Galois representations satisfying
\begin{itemize}
\item The $\br_i|_{\Gamma_{\Q(\zeta_{\ell})}}$ are irreducible  and pairwise non-isomorphic.
\item The determinants $\det(\br_i)$ are independent of $i$.
\item The $\br_i$ are all odd, i.e. $\det(\br_i(c))=-1$ for all $i$.
\end{itemize}
Writing $\br_i(g)= \begin{pmatrix} a_i(g) & b_i(g) \\ c_i(g) & d_i(g) \end{pmatrix}$, define
\[
\br(g)= \left( \begin{array}{cccc|cccc}
a_1(g) & 0 & \cdots& 0 & b_1(g) & 0 & \cdots & 0 \\ 0 & a_2(g) & 0 & \cdots & 0 & b_2(g) & 0 & \cdots \\
0 & 0 & \ddots & 0 & 0 & 0 & \ddots & 0\\
0 & \cdots & 0 & a_n(g) & 0 & \cdots & 0 & b_n(g) \\ \hline
c_1(g) & 0 & \cdots& 0 & d_1(g) & 0 & \cdots & 0 \\ 0 & c_2(g) & 0 & \cdots & 0 & d_2(g) & 0 & \cdots \\
0 & 0 & \ddots & 0 & 0 & 0 & \ddots & 0\\
0 & \cdots & 0 & c_n(g) & 0 & \cdots & 0 & d_n(g)
\end{array}
\right).
\]
The condition that $\det(\br_i)$ be independent of $i$ guarantees that
$\br$ is valued in $\mr{GSp}_{2n}$ and has symplectic multiplier
$\det(\br_i)$. Oddness of each $\br_i$ therefore implies oddness of
$\br$. Since the $\br_i|_{\Gamma_{\Q(\zeta_{\ell})}}$ are irreducible,
non-isotropic  and pairwise non-isomorphic, the only
  subspaces invariant under $\br$ are sums of some subset of the
  $\br_i$. It follows that $\br|_{\Gamma_{\Q(\zeta_{\ell})}}$ leaves
  no (nonzero) isotropic subspace invariant, so does not have image
  contained in any proper parabolic subgroup of $\mr{GSp}_{2n}$. Thus
$\br|_{\Gamma_{\Q(\zeta_{\ell})}}$ is $\mr{GSp}_{2n}$-irreducible,
while obviously not $\mr{GL}_{2n}$-irreducible.
\begin{prop}
Assume $\ell \gg_n 0$. With notation as above, $\br$ admits a Hodge--Tate regular de Rham lift $\rho \colon \Gamma_{\Q} \to \mr{GSp}_{2n}(\ov{\Z}_{\ell})$ with Zariski-dense image.
\end{prop}
\begin{proof}
From the above discussion, this will now follow from Theorem \ref{fkpthmA} provided we can check the local hypotheses in that theorem. For this, we note that by Lemma \ref{local} below each $\br_i|_{\Gamma_{\Q_{\ell}}}$ admits a potentially crystalline lift $\rho_{i, \ell}$ satisfying:
\begin{itemize}
\item $\det(\rho_{i, \ell})$ is independent of $i$.
\item The Hodge--Tate weights of each $\rho_{i, \ell}$ are distinct and
  also distinct as we vary $i$.
\end{itemize}
Let $\mu \colon \Gamma_{\Q} \to \ov{\Z}_{\ell}^\times$ be a global character such that $\mu|_{\Gamma_{\Q_{\ell}}}= \det(\rho_{i, \ell})$. Such a $\mu$ exists either by the local and global Kronecker--Weber theorems or by noting that $\det(\rho_{i, \ell})$ is an integer power of the cyclotomic character multiplied by a finite-order character; this reduces to the case of finite-order characters, where we can realize any finite extension of $\Q_{\ell}$ as a completion of a finite extension of $\Q$. For such a $\mu$, the local hypothesis at $\ell$ of Theorem \ref{fkpthmA} is satisfied by the choice of similitude character $\mu$ and the local lift $\rho_{\ell}$ constructed by summing the $\rho_{i, \ell}$ in the same way as $\br$ is defined in terms of the $\br_i$. For primes $p \neq \ell$ where $\br$ ramifies, \cite[Theorem 1.1]{booher:minimal} implies that $\br|_{\Gamma_{\Q_p}}$ admits a lift $\Gamma_{\Q_p} \to \mr{GSp}_{2n}(\ov{\Z}_{\ell})$ with similitude character $\mu$. We now conclude by Theorem \ref{fkpthmA}.
\end{proof}
Here is the local lemma used in the proof:
\begin{lem}\label{local}
Let $\br_1, \br_2, \ldots, \br_n \colon \Gamma_{\Q_{\ell}} \to \mr{GL}_2(\ov{\F}_{\ell})$ be continuous representations with the same determinant $\bar{\tau}= \det(\br_i)$, $i=1, \ldots, n$. Then there exist potentially crystalline lifts $\rho_i \colon \Gamma_{\Q_{\ell}} \to \mr{GL}_2(\ov{\Z}_{\ell})$ such that the union of the Hodge--Tate weights of the $\rho_i$ is a set with $2n$ distinct elements, and $\det(\rho_i)$ is independent of $i$.
\end{lem}
\begin{proof}
First we note that it suffices to produce potentially crystalline lifts $\rho_i$ with (all taken together) distinct Hodge--Tate weights and $\det(\rho_i)$ having the same (single) Hodge--Tate weight for all $i$. Indeed, then each quotient $\det(\rho_i)/\det(\rho_1)$ is a finite-order character valued in a pro-$\ell$ group (since the reduction mod $\ell$ is trivial), and so we can extract a square root and twist $\rho_i$ to have the same determinant as $\rho_1$ (and this finite-order twist does not affect the property of being potentially crystalline). To finish the proof, we apply \cite[Th\'{e}or\`{e}me 2.5.3, Theorem 2.5.4]{muller:thesis}, which show the following:
\begin{itemize}
\item If $\br_i$ is irreducible, then for any choice of Hodge--Tate weights $\{m_{i, 1}, m_{i, 2}\}$, there exists a potentially crystalline lift $\rho_i$ of $\br_i$ with these weights. We choose these in such a way that $m_{i, 1}+m_{i, 2}$ is independent of $i$, but the multi-set $\{m_{i, j}\}_{i, j}$ is in fact a set.
\item If $\br_i= \begin{pmatrix} \bar{\chi}_{i, 1} & * \\ 0 & \bar{\chi}_{i, 2} \end{pmatrix}$ is an extension of characters, then for each $i$ there are potentially crystalline lifts $\chi_{i, j}$ of $\bar{\chi}_{i, j}$ with Hodge--Tate weights summing to any pre-specified value and a potentially crystalline lift of $\br_i$ that is an extension of $\chi_{i, 2}$ by $\chi_{i, 1}$. (To see this claim requires inspection of the proof of \cite[Th\'{e}or\`{e}me 2.5.4]{muller:thesis}.)
\end{itemize}
\end{proof}
\section{Approach via the potential IGP with local conditions}\label{pIGP}
The general approach of the present section is quite flexible and will
certainly apply to other Galois images and target groups than those
used here; we do not strive for maximal generality. Let $N=2n \geq 4$
be an even integer, and consider the standard embedding
$\mr{SO}_N \to \mr{SO}_{N+1}$ of special orthogonal groups over
$\Z_{\ell}$ (defined, for definiteness, with respect to the symmetric
pairings given by the identity matrices\footnote{
    $\mr{SO}_N$ may not be split over $\Z_{\ell}$, but $\mr{SO}_{N+1}$
    is split since $N+1$ is odd.
  Many variants on the present construction are possible; in particular, we could take different forms of $\mr{SO}_N$.}). Let $\Gamma= \mr{SO}_N(\F_{\ell})$, and fix the following order two element $c_{\infty}$ of $\Gamma$:
\begin{itemize}
\item If $N \equiv 0 \pmod 4$, then $c_\infty= \begin{pmatrix} 0 & 1_n \\ 1_n & 0 \end{pmatrix}$ (note that $\det(c_{\infty})= (-1)^{n}=1$).
\item If $N \equiv 2 \pmod 4$, then $c_\infty= \left( \begin{array}{cc|cc} 0_{n-1} & 0 & 1_{n-1} & 0 \\
0 & -1 & 0 & 0 \\ \hline
1_{n-1} & 0 & 0_{n-1} & 0 \\
0 & 0 & 0 & -1
\end{array} \right)$ (note that $\det(c_{\infty})= (-1)^{n-1}= 1$).
\end{itemize}
We now circumvent the inverse Galois problem by applying
Moret-Bailly's solution (\cite{moret-bailly}) to the potential inverse
Galois problem with local conditions, and we thus produce the examples
of this section:
\begin{prop}\label{pIGPprop}
There exists a totally real field $F/\Q$ and a Galois extension $K/F$ satisfying:
\begin{enumerate}
\item There is an isomorphism $\bar{\theta} \colon \gal(K/F) \xrightarrow{\sim} \Gamma$.
\item $K/\Q$ is linearly disjoint from $\Q(\zeta_{\ell})/\Q$.
\item For all complex conjugations $c_v \in \gal(K/F)$ with
  $v \vert \infty$, $\bar{\theta}(c_v)$ is conjugate to $c_{\infty}$.
\item The prime $\ell$ splits completely in $F/\Q$, and for all places $v \vert \ell$ of $F$, and all $w \vert v$ of $K$, $K_w=F_v$ is the trivial extension.
\end{enumerate}
Assume now $\ell \gg_N 0$, and define $\br= \bar{\theta} \oplus 1$. Then $\br \colon \Gamma_F \to \mr{SO}_{N+1}(\F_{\ell})$ has a Hodge--Tate regular geometric lift $\rho \colon \Gamma_F \to \mr{SO}_{N+1}(\ov{\Z}_{\ell})$ with image Zariski-dense in $\mr{SO}_{N+1}$.
\end{prop}
\begin{proof}
Existence of $K$, $F$, and $\bar{\theta}$ follow immediately from \cite{moret-bailly}. To finish the proof, we must verify that the hypotheses of Theorem \ref{fkpthmA} are satisfied for $\br= \bar{\theta} \oplus 1$. A maximal (proper) parabolic subgroup of $\mr{SO}_{N+1}$ is the stabilizer of an isotropic subspace $W \subset \ov{\F}_{\ell}^{N+1}$. Since the image of $\bar{\theta}|_{\Gamma_{\Q(\zeta_{\ell})}}$ is $\mr{SO}_{N}(\F_{\ell})$, which acts absolutely irreducibly in its standard $N$-dimensional representation for $\ell \gg_N 0$, $\br$ stabilizes exactly two proper subspaces of $\ov{\F}_{\ell}^{N+1}$, namely the space of $\bar{\theta}$ and the complementary line. Clearly neither of these subspaces is isotropic, so in all cases $\br|_{\Gamma_{\Q(\zeta_{\ell})}}$ is $\mr{SO}_{N+1}$-absolutely irreducible. 

Booher's result (\cite[Theorem 1.1]{booher:minimal}) shows that for
$v$ not above $\ell$ at which $\br$ is ramified, $\br|_{\Gamma_{F_v}}$
has a lift to $\Gamma_{F_v} \to \mr{SO}_{N+1}(\ov{\Z}_{\ell})$. (Note
that Booher's result shows a $\mr{GO}_{N+1}$-deformation ring with
fixed orthogonal multiplier is formally smooth of suitably large
dimension; we fix the multiplier to be trivial to produce an
$\mr{O}_{N+1}= \mr{SO}_{N+1} \times \{\pm 1\}$ lift and then it is
clear that the image must lie in $\mr{SO}_{N+1}$.) Since the
extensions $K_w/F_v$ are trivial for $v \vert \ell$,
 $\br|_{\Gamma_{F_v}}$ has a crystalline Hodge--Tate regular lift by
Lemma \ref{lem:triv}. To conclude, we check
that $\br$ is odd. Indeed, since the adjoint representation of
$\mr{SO}_{N+1}$ is isomorphic to the second exterior power of the
standard representation, we find
\[
\tr(\br(c)|_{\mg{so}_{N+1}})= \frac{\tr(\br(c))^2- \tr(\br(c^2))}{2}= \frac{1-(N+1)}{2}= -\mr{rk}(\mr{SO}_{N+1}).
\]
This is equivalent to the oddness condition $\dim(\mg{so}_{N+1}^{\mr{Ad}(\br(c))=1})= \dim \mr{Flag}_{\mr{SO}_{N+1}}$ assumed in Theorem \ref{fkpthmA}.
\end{proof}

  \begin{lem} \label{lem:triv}%
  Let $K/\Q_p$ be a finite extension, $G$ a split reductive group over
  $\Z_p$ and $\br:\Gamma_K \to G(\F_p)$ the trivial
  representation. Then $\br$ has a Hodge--Tate regular crystalline
  lift $\rho:\Gamma_K \to G(\Z_p)$.
\end{lem}
\begin{proof}
  The $p$-adic cyclotomic character $\kappa_p:\Gamma_K \to \Z_p^*$ is
  crystalline. Therefore, so is $(\kappa_p)^{p-1}$ and it is moreover
  trivial modulo $p$. If $\chi:\G_m \to G$ is any regular cocharacter,
  then we may clearly take $\rho$ to be $\chi \circ (\kappa_p)^{p-1}$.
\end{proof}

\begin{rem}\label{notodd}
Perhaps the most interesting case here is when $N \equiv 2 \pmod 4$. Then the representations $\bar{\theta}$ are not odd, and indeed cannot be odd since $-1$ does not belong to the Weyl group of $\mr{SO}_N$ for any $N \equiv 2 \pmod 4$.\footnote{We learned from Wushi Goldring the observation that an $\mr{SO}_N$-representation that is not odd can transfer to an odd $\mr{SO}_{N+1}$-representation.} Existing lifting techniques cannot lift them to geometric $\mr{SO}_N$-representations, and we expect that no such lifts that are Hodge--Tate regular as $\mr{GL}_N$-representations can exist. See Proposition \ref{distinctHT} below for a proof of this in some cases, but we now explain a general heuristic. Such a lift (necessarily $\mr{GL}_N$-irreducible) would conjecturally arise from an automorphic representation $\pi$ of (split) $\mr{SO}_N/\Q$ with cuspidal transfer to $\mr{GL}_N$. The archimedean L-parameter $\mr{rec}_{\pi_{\infty}} \colon W_{\R} \to \mr{SO}_N(\C)$ would then (by archimedean purity for $\mr{GL}_N$) restrict to $\C^\times \subset W_{\R}$ as $\mr{rec}_{\pi_{\infty}}(z)= z^{\mu} \cdot \bar{z}^{-\mu}$ for some cocharacter $\mu$ of the diagonal torus $T \subset \mr{SO}_N$. Let $\{e_i^*\}_{i=1}^n$ be the standard basis of $X_*(T)$. Then the most regular situation that can arise has $\mr{rec}_{\pi_{\infty}}(j)$ equal to the element $c_{\infty}$ above, and $\mu= \sum_{i=1}^n p_i e_i^*$ with $p_1, \ldots, p_{n-1}$ distinct, and $p_n=0$ (using the Weil group relation $jzj^{-1}=\bar{z}$). Such an L-parameter is in fact $\mr{SO}_N$-regular, but it is clearly not $\mr{GL}_N$-regular. We note that these ``most regular" lifts that might be possible would be the $\ell$-adic representations associated to cuspidal automorphic representations on (suitable forms of) $\mr{SO}_N$ that are non-degenerate limits of discrete series at archimedean places. See \cite{goldring-koskivirta:stratahasse} for important recent progress on the problem of associating Galois representations to automorphic representations that are non-degenerate limits of discrete series at infinity.
\end{rem}
We can unconditionally rule out ``many" candidates for regular lifts of $\bar{\theta}$ using potential automorphy theorems:
\begin{prop}\label{distinctHT}
Let $\rho \colon \Gamma_{\Q} \to \mr{O}_N(\ov{\Z}_{\ell})$ for $\ell>2(2N+1)$ be a continuous representation satisfying:
\begin{enumerate}
\item $\rho$ is unramified outside a finite set of primes.
\item $\br|_{\Gamma_{\Q(\zeta_{\ell})}}$ is $\mr{GL}_N$-irreducible.
\item $\rho|_{\Gamma_{\Q_\ell}}$ is as $\mr{GL}_N$-representation potentially diagonalizable in the sense of \cite[\S 1.4]{blggt:potaut} and has distinct Hodge--Tate weights.
\end{enumerate}
Then $\tr(\rho(c))=0$. In particular, when $N \equiv 2 \pmod 4$, $\rho$ cannot factor through $\mr{SO}_N$.
\end{prop}
\begin{proof}
  In this proof we freely use the terminology and notation of
  \cite{blggt:potaut}; we argue as in, for instance, \cite[Proposition
  3.3.1]{blggt:potaut}, making use of Harris's tensor product
  trick. Choose a quadratic imaginary field $K/\Q$ linearly disjoint
  from $\Q(\br, \zeta_{\ell})/\Q$, and let
  $\psi \colon \Gamma_K \to \ov{\Z}_{\ell}^\times$ be a geometric
  (hence potentially crystalline and potentially diagonalizable)
  character such that $r:= \mr{Ind}_{\Gamma_K}^{\Gamma_{\Q}}(\psi)$
  satisfies:
\begin{itemize}
\item $r^* \cong r \otimes \mu$, for some character
    $\mu$ with $\mu(c)=-1$ (this imposes no condition on $\psi$).
\item $r|_{\Gamma_{\Q_{\ell}}}$ is Hodge--Tate regular, with Hodge numbers $\{m, 0\}$ having the property that no two Hodge numbers of $\rho$ differ by $|m|$.
\item $(\br \otimes \bar{r})|_{\Gamma_{\Q(\zeta_{\ell})}}$ is irreducible. 
\end{itemize}
To produce such a $\psi$, we can apply \cite[Lemma A.2.5]{blggt:potaut} as in \cite[Proposition 4.1.1]{blggt:potaut}. Namely, let $m$ be some integer such that no two Hodge numbers of $\rho$ differ by $m$, let $q$ be a rational prime split in $K(\zeta_{\ell})$ at which $\br$ is unramified, let $\ov{\chi}_0 \colon \Gamma_{\Q} \to \ov{\mathbb{F}}_{\ell}^\times$ be any finite-order character ramified at $q$, with Teichm\"{u}ller lift $\chi_0 \colon \Gamma_{\Q} \to \ov{\Z}_{\ell}^\times$, and let $\chi= \kappa_{\ell}^{m} \chi_0$, where $\kappa_{\ell}$ denotes the $\ell$-adic cyclotomic character. At the two places $v$ and $\bar{v}$ of $K$ above $\ell$, define $\psi_v= \chi|_{\Gamma_{K_v}}$ and $\psi_{\bar{v}}= 1$ (as characters of $\Gamma_{K_v}$ and $\Gamma_{K_{\bar{v}}}$, respectively). At the two places $v_q$ and $\bar{v}_q$ of $K$ above $q$, define $\psi_{v_q}= \chi|_{\Gamma_{K_{v_q}}}$ and $\psi_{\bar{v}_q}=1$. Letting $S=\{v, \bar{v}, v_q, \bar{v}_q\}$, we have that for each $w \in S$, $\psi_w \psi_{\bar{w}}^c|_{I_{K_w}}= \chi|_{I_{K_w}}$, so \cite[Lemma A.2.5]{blggt:potaut} implies there exists a character $\psi \colon \Gamma_K \to \ov{\Z}_{\ell}^\times$ such that $\psi \psi^c= \chi|_{\Gamma_K}$ and $\psi|_{I_{K_w}}= \psi_w|_{I_{K_w}}$ for all $w \in S$. The resulting $r= \mr{Ind}_{\Gamma_K}^{\Gamma_{\Q}} (\psi)$ satisfies the three desired properties: the first two are clear, and we now check the irreducibility. By assumption, $\bar{\rho}|_{\Gamma_{K(\zeta_{\ell})}}$ is irreducible, so if $(\br \otimes \bar{r})|_{\Gamma_{\Q(\zeta_{\ell})}}$ failed to be irreducible, it would have a constituent $\ov{W}$ restricting to either $\bar{\rho}\bar{\psi}|_{\Gamma_{K(\zeta_{\ell})}}$ or $\bar{\rho}\bar{\psi}^c|_{\Gamma_{K(\zeta_{\ell})}}$. Since $\br \bar{\psi}|_{\Gamma_{K(\zeta_{\ell})}}$ is ramified at (places above) $v_q$ but not at (places above) $\bar{v}_q$, neither $\br \bar{\psi}|_{\Gamma_{K(\zeta_{\ell})}}$ nor $\br \bar{\psi}^c|_{\Gamma_{K(\zeta_{\ell})}}$ can extend to $\Gamma_{\Q(\zeta_{\ell})}$, establishing the claim.

Consider then the representation $\rho':= \rho \otimes r \colon \Gamma_{\Q} \to \mr{GL}_{2N}(\ov{\Z}_{\ell})$. Since $\rho$ is orthogonal, and $r$ is symplectic, $\rho' \colon \Gamma_{\Q} \to \mr{GSp}_{2N}(\ov{\Z}_{\ell})$ is symplectic. Moreover, $(\rho')^* \cong \rho' \otimes \mu$, and thus we see that $\rho'$ is totally odd polarizable. Moreover, $\rho'|_{\Gamma_{\Q_{\ell}}}$ is still potentially diagonalizable with distinct Hodge--Tate weights. Thus we can apply the potential automorphy theorem of \cite[Corollary 4.5.2]{blggt:potaut} to conclude that $(\rho', \mu)$ is potentially automorphic. 
Now we apply \cite[Lemma 2.2.4]{blggt:potaut}: in the notation of \textit{loc. cit.}, $(\rho'= \rho \otimes r \cong \mr{Ind}_{\Gamma_K}^{\Gamma_{\Q}}( \rho|_{\Gamma_K} \otimes \psi), \mu)$ is potentially automorphic and polarized, so $(\rho|_{\Gamma_K} \otimes \psi, \mu)$ is also potentially automorphic and polarized. By \cite[Lemma 2.2.1, Lemma 2.2.2]{blggt:potaut}, $(\rho, 1)$ is potentially automorphic, i.e. there exists a regular algebraic, polarized cuspidal automorphic representation $(\pi, \chi)$ of $\mr{GL}_N(\A_{\Q})$ such that $(\rho, 1) \cong (r_{\ell, \iota}(\pi), \kappa_{\ell}^{1-N}r_{\ell, \iota}(\chi))$,
and $r_{\ell, \iota}(\pi)$ is the automorphic Galois representation associated to $\pi$ (see \cite[Theorem 2.1.1]{blggt:potaut}; note that in their normalization, $r_{\ell, \iota}(\pi)$ is the Galois representation whose local restrictions correspond under local Langlands to $\pi \otimes |\cdot|^{\frac{1-N}{2}}$). It follows that $\chi= |\cdot|^{N-1}$, and $\pi \otimes |\cdot|^{\frac{1-N}{2}}$ is the (on the nose) self-dual regular L-algebraic cuspidal automorphic representation corresponding to $\rho$ under the local Langlands correspondence. Applying \cite[Theorem A]{taibi:trace} to $\pi \otimes |\cdot|^{\frac{1-N}{2}}$, we find that $|\tr(\rho(c))| \leq 1$, and thus ($N$ is even) $\tr(\rho(c))=0$.
\end{proof}
\begin{rem}
The main limitation in this result is the potential diagonalizability assumption. But in a compatible system of Hodge--Tate regular automorphic Galois representations, almost all members will be potentially diagonalizable by the theory of Fontaine--Laffaille. Calegari (\cite{calegari:even2}) has proven a stronger result that a geometric and Hodge--Tate regular $\Gamma_{\Q} \to \mr{GL}_2(\ov{\Z}_{\ell})$ must be odd, without a potential diagonalizability hypothesis.
\end{rem}
\section{Zywina's examples}\label{zywinasection}

\subsection{A review of Zywina's work}

In this section we recall the construction of Zywina (\cite{zywina:orthogonal}) that in many cases realizes the simple groups of orthogonal type over $\F_\ell$ as Galois groups of regular extensions of $\Q(t)$, and in particular as Galois groups over $\Q$ (in infinitely many ways). Some of the finer points--namely, the simplicity itself--of this construction will not matter for our present purposes: rather, we want to use this construction as a source of Galois groups over $\Q$ isomorphic to large subgroups of even orthogonal groups, and we recall only what is necessary for our purposes. 

Fix an even integer $N \geq 6$. Let $R= \Z[S^{-1}]$, for a finite set of primes $S$ that may be enlarged as the construction proceeds. Zywina first considers polynomials $a_2(t), a_4(t), a_6(t) \in R[t]$, such that the discriminant $\Delta(t)$ of the Weierstrass equation $y^2=x^3+a_2(t)x^2+a_4(t)x+a_6(t)$ is non-zero, and the $j$-invariant is non-constant, and builds a family of quadratic twists of this Weierstrass equation. This leads to the following construction of a rank $N$ orthogonal local system on an open subset of $\P^1_{R}$, which we rapidly summarize:
\begin{itemize}
\item Let $A=R[u, \Delta(u)^{-1}])$ and set $M= \Spec(A)$.
\item Let $j \colon U \to \P_M^1$ be the inclusion of the open $M$-subscheme defined by $U= \Spec(A[t, (t-u)^{-1}, \Delta(t)^{-1}])$.
\item Let $E \to U$ be the elliptic curve defined by the Weierstrass equation
\[
(t-u)y^2= x^3+a_2(t)x^2+a_4(t)x+a_6(t),
\]
\item Let $\pi \colon \P^1_M \to M$ be the structure morphism, and define
\[
\mc{G}= R^1 \pi_*(j_* E[\ell]),
\]
where $E[\ell]$ is the local system (of $\F_{\ell}$-modules) on $U$ defined by the $\ell$-torsion subgroup scheme of $E$. The sheaf $\mc{G}$ on $M$ is clearly constructible, and Zywina shows (\cite[Lemma 3.3]{zywina:orthogonal}) that it is lisse after suitable enlargement of $S$. Specializing to a point $m$ of $M$ over a finite field $k$ yields an elliptic curve $E_m$ over the global field $k(t)$, and Zywina shows the rank of $\mc{G}$ is independent of the specialization and can be computed explicitly in terms of bad reduction invariants of $E_m$ (\cite[Proposition 2.8, Lemma 3.1]{zywina:orthogonal}); precisely, the rank $N$ of $\mc{G}$ is $-4+\deg(\mr{cond}(E_m))$, where $\deg(\mr{cond}(E_m))$ is the degree of the conductor of $E_m$. By a careful choice of the $a_i(t)$ and case-by-case calculation with Tate's algorithm (in \cite[\S 6]{zywina:orthogonal}), Zywina produces examples with our desired even rank $N \geq 6$. Poincar\'{e} duality provides an orthogonal pairing $\mc{G} \times \mc{G} \to \F_{\ell}$, and so given a geometric generic point $\bar{\xi}$ of $M$ we obtain a representation 
\[
\theta_{\ell} \colon \pi_1(M, \bar{\xi}) \to \mr{O}(\mc{G}_{\bar{\xi}}).
\]
\end{itemize}
In fact Zywina must, in order to realize the simple groups as Galois groups, work with the pull-back $h^*(\mc{G})$ of this local system along a suitable finite \'{e}tale cover $h \colon W \to M$, where $W$ is again an open subscheme of $\P^1_R$ (see the proof of \cite[Theorem 4.1]{zywina:orthogonal}). He denotes by 
\[
\vartheta_{\ell} \colon \pi_1(W) \to \mr{O}(V_{\ell})
\] 
the representation associated to $h^*(\mc{G})$ (omitting reference to the choice of compatible base-point).

Recall (see \cite[\S 1.1]{zywina:orthogonal}) that the spinor norm is a homomorphism $\mr{sp} \colon \mr{O}(V_{\ell}) \to \F_{\ell}^\times/(\F_{\ell}^\times)^2$, and let $\Omega(V_{\ell}) \subset \mr{SO}(V_{\ell})$ be the subgroup of elements with trivial spinor norm. When the discriminant $\mr{disc}(V_{\ell}):= \mr{sp}(-1)$ is the trivial coset, $-1$ belongs to $\Omega(V_{\ell})$. We will use the following consequence of Zywina's main theorem and arguments:
\begin{prop}\label{hilbert}
Let $N \geq 6$ be an even integer, and let $\ell \geq 5$ be a prime. Then the elliptic curve $E \to U$ may be chosen to ensure that the quadratic space $V_{\ell}$ has discriminant $(\F_{\ell}^\times)^2$, and that there are infinitely many $w_i \in \Q$ such that the specializations $\vartheta_{\ell, w_i}$ (defined up to conjugation)
\[
\Gamma_{\Q} \xrightarrow{w_i} \pi_1(W) \xrightarrow{\vartheta_{\ell}} \mr{O}(V_{\ell})
\]
are non-isomorphic and satisfy $\vartheta_{\ell, w_i}(\Gamma_{\Q})= \Omega(V_{\ell})$.

When $N \equiv 6 \pmod 8$, we also consider the same family $E \to U$ and the pull-back along a different $h \colon W \to M$, and the above holds except with the conclusion that $\vartheta_{\ell, w_i}(\Gamma_{\Q})$ contains $\Omega(V_{\ell})$ with index 2, and $\vartheta_{\ell, w_i}(\Gamma_{\Q(i)})$ equals $\Omega(V_{\ell})$. We will refer to the two examples when $N \equiv 6 \pmod 8$ as Case $6_\Omega$ and Case $6_{\mr{O}}$.
\end{prop}
\begin{proof}
  When $N \equiv 0, 2, 4 \pmod{8}$, or in Case $6_\Omega$, the
  existence of specializations $w_i$ with
  $\vartheta_{\ell, w_i}(\Gamma_\Q)= \Omega(V_{\ell})$ is immediate
  from the proof of \cite[Theorem 1.1]{zywina:orthogonal} (in \S 6 of
  \emph{op.~cit.}) and the Hilbert irreducibility theorem (\cite[\S
  3.3-3.4]{serre:topics}), which produces $w_i$ such that the fixed
  fields of the $\vartheta_{\ell, w_i}$ are linearly disjoint over
  $\Q$. Case $6_{\mr{O}}$ arises from an earlier version (\cite{zywina:orthogonalold}) of
  \cite{zywina:orthogonal}. Because that version is not destined to be published, we rapidly recall the family $E$ considered and its numerical invariants, which are all computed by Tate's algorithm. We follow the notation of \cite{zywina:orthogonal}. Writing $N=8n+6$, define $h(u)= \frac{u^2+1}{2u}$ and $f(t)= \prod_{i=1}^{4n+1} (t-h(i+1))$, and consider the Weierstrass equation (defining the polynomials $a_i(t)$ in this case)
  \[
  y^2= (x-t(t^2-2)f(t))\cdot (x-t(t^2+1)f(t))\cdot (x+t(2t^2-1)f(t)).
  \]
  For any sufficiently large prime $p$ and any $w \in W(\mathbb{F}_p)$, the points of bad reduction of the elliptic curve $E_{h(w)}$ over $\mathbb{F}_p(t)$ are $x=1, -1, 0, \infty$, and any root $x=a$ of $(t-h(w))f(t) \pmod p$. The bad reduction invariants are as follows:
  
  \begin{center}
  \begin{tabular}{| c | c | c | c | }
  \hline
  $x$ & $\kappa_x$ & split multiplicative reduction iff & $c_x(E_{h(w)})$ \\ \hline
  -1 & $\mr{I}_2$ & $-3(1+h(w))f(-1) \in (\mathbb{F}_p^\times)^2$ & 2 \\ \hline
  1 & $\mr{I}_2$ & $-3(1-h(w))f(1) \in (\mathbb{F}_p^\times)^2$ & 2\\ \hline
  0 & $\mr{I}_4^*$ & $\emptyset$ & $c_0$\\ \hline
  $\infty$ & $\mr{I}_4^*$ & $\emptyset$ & $c_{0}$ \\ \hline
  $a$ & $\mr{I}_0^*$ & $\emptyset$ & 4\\ \hline
  \end{tabular}
  \end{center}
  
  (The precise value $c_0$ does not matter here: it will only matter that these invariants are the same at $0$ and $\infty$.) From this bad reduction data, one easily computes the other invariants of the curve $E_{h(w)}$ as in \cite[\S 2]{zywina:orthogonal}; we will invoke these calculations below. The calculations of \cite[\S 6.4]{zywina:orthogonalold} are correct except the root number
  \[
  \varepsilon_{E_{h(w)}}= \left(\frac{-3(1+h(w))f(-1)}{p} \right)\left( \frac{-3(1-h(w))f(1)}{p} \right) \left( \frac{-1}{p} \right)^2 \left(\frac{-1}{p} \right)^{\deg(f)+1}= \left( \frac{-1}{p} \right)
  \]
  (in \textit{loc.~cit.}, the last exponent $\deg(f)+1$ is replaced with $\deg(f)$, so there is an incorrect conclusion that the root number is 1). Using this corrected root number, it follows from \cite[Proposition 3.2]{zywina:orthogonal} that $\det(\vartheta_{\ell})(\mr{Frob}_w) = \left( \frac{-1}{p} \right)$
  for all $p \gg 0$, $w \in W(\F_p)$, and so
  $\det(\vartheta_{\ell}) \colon \pi_1(W) \to \F_{\ell}^\times$
  factors through the non-trivial quadratic character
  $\pi_1(W) \to \Gamma_{\Q} \to \gal(\Q(i)/\Q) \to \F_{\ell}^\times$
  (with the first map induced by the structure morphism). The proof of
  \cite[Theorem 4.1]{zywina:orthogonal} then still shows that the geometric
  monodromy group $\vartheta_{\ell}(\pi_1(W_{\ov{\Q}}))$ contains
  $\Omega(V_{\ell})$, and that $\vartheta_{\ell}(\pi_1(W_{\Q(i)}))$ is
  contained in $\Omega(V_{\ell})$.

  Indeed, first we note that the calculation of
  $\det(\vartheta_{\ell})$ shows $\vartheta_{\ell}(\pi_1(W_{\Q(i)}))$
  is contained in $\mr{SO}(V_{\ell})$. Next we observe that
  \cite[Theorem 3.4]{zywina:orthogonal} (originally due to Hall in
  \cite{hall:bigmonodromy}) shows that the image of
  $\pi_1(M_{\ov{\Q}})$ contains $\Omega(V_{\ell})$---this is easily
  read off from the bad reduction invariants tabulated above---and
  then the observation in the proof of \cite[Theorem
  4.1]{zywina:orthogonal} that $\vartheta_{\ell}(\pi_1(W_{\ov{\Q}}))$
  also contains $\Omega(V_{\ell})$ continues to apply (namely, that
  the morphism $h$ has degree at most four---in fact two in this
  case---so that a Galois closure of $W_{\ov{\Q}} \to M_{\ov{\Q}}$ has
  solvable Galois group, whereas $\Omega(V_{\ell})$ is simple). We
  conclude by noting that $\vartheta_{\ell}(\pi_1(W_{\Q(i)}))$ is
  contained in $\Omega(V_{\ell}) \subset \mr{SO}(V_{\ell})$: this
  follows as in \cite[Theorem 4.1]{zywina:orthogonal} from the bad
  reduction invariants in the table (again, this part of the
  calculation does not depend on the erroneous root number in \cite[\S
  6.4]{zywina:orthogonalold}).
  % only
  %uses information about the Euler characteristic $\chi$ and the
  %Tamagawa number $c_{E_{h(w)}}$, in the notation of
  %\textit{loc.~cit.}). 
  Thus $\vartheta_{\ell}(\pi_1(W_{\ov{\Q}}))= \vartheta_{\ell}(\pi_1(W_{\Q(i)}))= \Omega(V_{\ell})$, and
  $\vartheta_{\ell}(\pi_1(W)) \subset \mr{O}(V_{\ell})$ contains $\Omega(V_{\ell})$ with
  index 2. We can again invoke Hilbert irreducibility to produce the
  desired $u_i$: there is a thin set $T \subset W(\Q)$ such that for
  all $w \in W(\Q) \setminus T$, $\vartheta_{\ell, w}$ cuts out a
  Galois extension of $\Q$ with Galois group isomorphic to
  $\vartheta_{\ell}(\pi_1(W))$. Since
  $W(\Q) \setminus T \subset \Q(i)$ is not a thin subset (see
  \cite[Proposition 3.2.1]{serre:topics}), applying Hilbert
  irreducibility to $\vartheta_{\ell}|_{\pi_1(W_{\Q(i)})}$ shows that
  there are infinitely many $w_i \in W(\Q) \setminus T$ such that
  $\vartheta_{\ell, w_i}(\Gamma_{\Q(i)})$ equals
  $\Omega(V_{\ell})$. The resulting extensions are linearly disjoint
  over $\Q(i)$.
\end{proof}
\begin{rem}
We have included the additional example of Case $6_{\mr{O}}$ to indicate that our results do not depend on all of the fine details of \cite{zywina:orthogonal}; one could imagine producing many more examples starting from Hall's big monodromy result (\cite[Theorem 3.4]{zywina:orthogonal}).
\end{rem}

\subsection{The trace of complex conjugation}

Let $w \in W(\R)$ be any real point, giving rise (fixing a $\C$-valued
geometric point $\bar{w}$ over $w$) to the stalk
$h^*(\mc{G})_{\bar{w}} \cong \mc{G}_{h(\bar{w})} \cong H^1(\P^1_{\C},
j_* E_{h(w)}[\ell])$ (we omit indicating the $h(w)$-specialization in
the notation for this $\P^1$ and $j$ so as not to burden the
notation). A minimal proper regular model
$\pi \colon E_{h(w)} \to \P^1_{\R}$ of the Weierstrass equation (over
$\R(t)$) defines our elliptic surface. The map $\pi$ is smooth over an
open subset $U_{h(w)} \subset \P^1_{\R}$ that depends on $w$.

In our application of the lifting theorem of \cite{fkp:reldef}, we
most importantly have to understand the action of complex conjugation
in the Galois representations $\vartheta_{\ell, w}$, $w \in
W(\Q)$. The Betti-\'{e}tale comparison isomorphism, applied to the
cohomology of $E_{h(w)}$, reduces this to a topological computation
over $\mathbb{C}$. In order not to interrupt the flow, we perform this
computation in Appendix \ref{complexconj}, and using this we prove the
following:
\begin{lem} \label{lem:z} %
  If $\ell \gg_N 0$, the trace of $c \in \gal(\C/\R)$ acting on the
  $\F_{\ell}$-vector spaces $H^1(\P^1_{\C}, j_* E_{h(w)}[\ell])$ is
  given by
\begin{equation}
  \mr{Tr}(c) = \begin{cases}
    -2 & \text{ if } N \equiv 2 \pmod 8 \\
    0  & \text{ if } N \equiv 4 \pmod 8 \\
    0 & \text { if } N \equiv 6 \pmod 8 \text{ (Case $6_{\mr{O}}$)}\\
    2 & \text { if } N \equiv 6 \pmod 8 \text{ (Case $6_{\Omega}$)}\\
    0 & \text { if } N \equiv 0 \pmod 8
  \end{cases}
\end{equation}
\end{lem}  
\begin{proof}
  We will deduce the lemma from Proposition \ref{prop:z}.  The Cases
  (1)-(4) therein correspond to the singular fibres, and their
  configurations over $\R$, of the elliptic surfaces considered in
  \cite[\S 6]{zywina:orthogonal} and (for Case (3${}_{\mr{O}}$))
  \cite[\S 6.4]{zywina:orthogonalold} as follows:
\begin{itemize}
\item Case (1) describes the fibres when $N \equiv 2 \pmod 8$;
\item Case (2) when $N \equiv 4 \pmod 8$;
\item Case (3${}_{\mr{O}}$) when $N \equiv 6 \pmod 8$, and we are in
  Case $6_{\mr{O}}$, and Case (3${}_{\Omega}$) when we are in Case
  $6_\Omega$;
\item Case (4) when $N \equiv 0 \pmod 8$.
\end{itemize}
This claim about the structure of the singular fibres follows almost
immediately from the descriptions in \cite[\S
6]{zywina:orthogonal}. The one point to note is that Tate's algorithm
shows that Zywina's description of when the $I_n$ fibres are split or
non-split is also valid over $\R$ (e.g., for $N \equiv 2 \pmod 8$, the
$I_1$ fibre at $\infty$ is split if and only if $-3$ is a square in
$\R$, hence it is non-split).

We now claim that, assuming $\ell$ is large enough as in \cite[\S
2.5]{zywina:orthogonal}, the traces $\tr(\mr{F}_{\infty}| V)$ recorded
in Proposition \ref{prop:z} via analysis of the cohomology groups
$H^1(\P^1_{\C}, j_* R^1 \pi_* \Q_{\pi^{-1}(U_{h(w)})})$ are the
negatives of the traces of $c \in \gal(\C/\R)$ on the
$\F_{\ell}$-vector spaces $H^1(\P^1_{\C}, j_* E_{h(w)}[\ell])$.  The
formulae in the lemma then follow from the claim and the formulae in
Proposition \ref{prop:z}.

To prove the claim, note that \cite[Equation
(2.1)]{zywina:orthogonal} shows that there is a short exact
sequence\footnote{Zywina asserts this when the ground field is
  $\ov{\mathbb{F}}_p$, but it holds equally well over the complex
  numbers. It uses the assumption on $\ell$ and the description of the
  singular fibers: exactness of
  $0 \to j_*T_{\ell}(E) \to j_* T_{\ell}(E) \to j_* E[\ell] \to 0$ is
  equivalent to there being no $\ell$-torsion in
  $(R^1j_* T_{\ell}(E))_x= T_{\ell}(E)/(\gamma_x-1)T_{\ell}(E)$, where
  $\gamma_x$ is the local monodromy at the point of bad reduction
  $x$. Using the dictionary between types in the Kodaira
  classification and local monodromy matrices, this follows from the
  assumption that $\ell \neq 2,3$ and that $\ell$ does not divide
  $\mr{ord}_x(j(E))$ when this valuation is negative, i.e. in the
  multiplicative reduction case. To pass from the above short exact
  sequence of sheaves to the desired short exact sequence of
  cohomology groups follows from the global monodromy calculation
  \cite[Proposition 2.7]{zywina:orthogonal}, which is also valid over
  $\mathbb{C}$.}
\[
0 \to H^1(\P^1_{\C}, j_* T_{\ell}(E)) \xrightarrow{\cdot \ell} H^1(\P^1_{\C}, j_* T_{\ell}(E)) \to H^1(\P^1_{\C}, j_* E[\ell]) \to 0,
\]
so $H^1(\P^1_{\C}, j_* T_{\ell}(E))$ is a free $\Z_{\ell}$-module of
rank equal to the $\F_{\ell}$-dimension of
$H^1(\P^1_{\C}, j_* E[\ell])$. Since $\ell \neq 2$, the eigenvalues of
$c \in \gal(\C/\R)$ are the same on each space, and we can therefore
compute $\tr(c|H^1(\P^1_{\C}, j_* E[\ell]))$ by computing
$\tr(c|H^1(\P^1_{\C}, j_* T_{\ell}(E)\otimes \Q_{\ell}))$. This Galois
module is isomorphic to
$H^1(\P^1_{\C}, j_* R^1 \pi_* \Q_{\ell, \pi^{-1}(U_{h(w)})})^\vee(-1)$
by Verdier duality; dualizing leaves $\tr(c)$ unchanged, and the $-1$
Tate twist multiplies $\tr(c)$ by $-1$, proving the claim, hence also
the lemma.

\end{proof}

\section{Main theorem}
It is now a simple matter to prove our main result:

\begin{thm} \label{thm:main}%
Let $N \geq 6$ be an even integer, and let $\ell \gg_N 0$ be a sufficiently large prime. 
\begin{enumerate}
\item There is a totally real field $F$ which is solvable over $\Q$ and infinitely many non-isomorphic Galois representations
\[
\br \colon \Gamma_F \to \mr{SO}_{N+1}(\ov{\F}_{\ell})
\]
such that $\br$ is irreducible as an $\mr{SO}_{N+1}$-valued representation, but reducible as a $\mr{GL}_{N+1}$-valued representation, and $\br$ admits a geometric lift $\rho \colon \Gamma_F \to \mr{SO}_{N+1}(\ov{\Z}_{\ell})$ with Zariski-dense image.
\item Moreover, if we assume that representations
  $\Gamma_{\Q_{\ell}} \to \mr{SO}_{N+1}(\ov{\F}_{\ell})$ admit
  Hodge--Tate regular de Rham lifts
  $\Gamma_{\Q_{\ell}} \to \mr{SO}_{N+1}(\ov{\Z}_{\ell})$, then in part
  (1) we may take $F=\Q$.
\end{enumerate}
\end{thm}

\begin{proof}
  Consider one of the specializations
  $\vartheta_{\ell, w_i} \colon \Gamma_{\Q} \to \mr{O}(V_{\ell})$ from
  Proposition \ref{hilbert}, where $V_{\ell}$ is a quadratic space
  over $\F_{\ell}$ of rank $N$ and trivial discriminant. We have four
  possibilities for these representations as in the table below, where
  we use Lemma \ref{lem:z} for the trace computations.  In each case,
  using the natural inclusion of $\mr{SO}_N$ in $\mr{SO}_{N+1}$, we
  define
  $\br'_{w_i} \colon \Gamma_{\Q} \to \mr{O}_{N+1}(\ov{\F}_{\ell}) =
  \mr{SO}_{N+1}(\ov{\F}_{\ell}) \times \{\pm 1\}$ (note that $N+1$ is
  odd) as follows, letting $\delta_{K/\Q}$ be the non-trivial
  quadratic character of an imaginary quadratic field linearly
  disjoint from the fixed field of $\vartheta_{\ell, w_i}$:

\begin{center}

\begin{tabular}{|c |c |c |c |}
\hline
$N \pmod 8$ & $\vartheta_{\ell, w_i}(\Gamma_{\Q})$ & $\tr(\vartheta_{\ell, w_i}(c))$ & $\br'_{w_i}$ \\ \hline
0 & $\Omega(V_{\ell})$ & 0 %$\begin{cases} 0 & \text{if the $III^*$ fibre is split,} \\ -2 & \text{otherwise} \end{cases}$ 
& $\vartheta_{\ell, w_i} \oplus 1$ \\ \hline
2 & $\Omega(V_{\ell})$ & -2 & $\vartheta_{\ell, w_i}  \oplus 1 $ \\ \hline
4 & $\Omega(V_{\ell})$ & 0 & $\vartheta_{\ell, w_i} \oplus 1$ \\ \hline
6, Case $6_{\Omega}$ & $\Omega(V_{\ell})$ & 2 & $(\delta_{K/\Q} \otimes \vartheta_{\ell, w_i}) \oplus 1$ \\ \hline
6, Case $6_{\mr{O}}$ & \tiny{$\Omega(V_{\ell}) \subsetneq \vartheta_{\ell, w_i}(\Gamma_{\Q}) \subsetneq \mr{O}(V_{\ell})$} & 0 & $\vartheta_{\ell, w_i} \oplus 1$ \\ \hline

\end{tabular}

\end{center}

We let $\br_{w_i}$ be the projection to the $\mr{SO}_{N+1}$-component;
of course, $\br_{w_i}= \br'_{w_i}$ except in Case $6_{\mr{O}}$. Each
$\br_{w_i}$ is clearly reducible as
$\mr{GL}_{N+1}(\ov{\F}_{\ell})$-representation, and we claim that even
after restriction to $\Gamma_{\Q(i, \zeta_{\ell})}$ each $\br_{w_i}$
is irreducible as
$\mr{SO}_{N+1}(\ov{\F}_{\ell})$-representation. First note that since
$\Omega(V_{\ell})/\{\pm 1\}$ is a non-abelian simple group,
$\vartheta_{\ell, w_i}(\Gamma_{\Q(i, \zeta_{\ell})})$ also equals
$\Omega(V_{\ell})$. A maximal (proper) parabolic subgroup of
$\mr{SO}_{N+1}$  (as a group over $\ov{\F}_{\ell}$) is
the stabilizer of an isotropic subspace
$W \subset \ov{\F}_{\ell}^{N+1}$. Since in all cases the image
$\vartheta_{\ell, w_i}(\Gamma_{\Q(i, \zeta_{\ell})})$ is
$\Omega(V_{\ell})$, $\br_{w_i}$ stabilizes exactly two proper
subspaces of $\ov{\F}_{\ell}^{N+1}$, namely $V_{\ell}$ itself and the
complementary line (the standard representation
$\Omega(V_{\ell}) \to \mr{SO}(V_{\ell} \otimes \ov{\F}_{\ell})$ is
irreducible for $\ell \gg_N 0$). Clearly neither of these subspaces is
isotropic, so in all cases $\br_{w_i}|_{\Gamma_{\Q(i, \zeta_{\ell})}}$ is
absolutely irreducible  as an $\mr{SO}_{N+1}$-valued
  representation. 

Each $\br_{w_i}$ is odd by the same calculation as in Proposition
\ref{pIGPprop}: this is what demands in Case $6_{\Omega}$
incorporating the twist by $\delta_{K/\Q}$.

Finally, to apply Theorem \ref{fkpthmA}, we have to check a local
lifting hypothesis on $\br_{w_i}$; here is where we will replace $\Q$
by a suitable totally real field. Namely, we do not know at present
that any $\br_{w_i}|_{\Gamma_{\Q_{\ell}}}$ admits a Hodge--Tate
regular de Rham lift to $\mr{SO}_{N+1}(\ov{\Z}_{\ell})$, so we
circumvent this problem by passing to a finite extension. For each
$w_i$, $\br_{w_i}|_{\Gamma_{\Q_{\ell}}}$ cuts out a finite extension
of $\Q_{\ell}$, and as $w_i$ varies these extensions have bounded
degree, so their composite $L/\Q_{\ell}$ is still finite (and
solvable). There is a solvable totally real extension $F/\Q$, linearly
disjoint from $\Q(i, \zeta_{\ell})$,
%, and from $K(i, \zeta_{\ell})$ when $N \equiv 2 \pmod 8$, 
such that for all primes $v$ of $F$ above $\ell$, the extension
$F_v/\Q_{\ell}$ is isomorphic to $L/\Q_{\ell}$. As $\Omega(V_{\ell})$
has no proper abelian quotient, it follows easily that such an $F$ is
linearly disjoint from $\Q(\br_{w_i})$ for all $w_i$.\footnote{A
  slightly simpler version of this argument would choose the extension
  $F$ depending on $w_i$. This would also enable us to avoid invoking
  Booher's theorem below, by also choosing $F= F(w_i)$ to trivialize
  $\br_{w_i}$ at all primes of ramification.} The required
irreducibility (by linear disjointness) and oddness still hold for
$\br_{w_i}|_{\Gamma_F}$, but now for all places $v \vert \ell$ of $F$,
$\br_{w_i}|_{\Gamma_{F_v}}$ is trivial, and so
$\br_{w_i}|_{\Gamma_{F_v}}$ admits a Hodge--Tate regular crystalline
lift $\Gamma_{F_v} \to \mr{SO}_{N+1}(\Z_{\ell})$  by Lemma \ref{lem:triv}.
% , simply by taking a suitable sum of powers of finite-order twists of the cyclotomic character.
At all places $v$ not above $\ell$ at which $\br_{w_i}$ is ramified, there exists a lift $\Gamma_{F_v} \to \mr{SO}_{N+1}(\ov{\Z}_{\ell})$ by \cite[Theorem 1.1]{booher:minimal}. We have therefore satisfied all of the hypotheses of Theorem \ref{fkpthmA}, and so for all $w_i$ there are Hodge--Tate regular geometric lifts $\rho_{w_i} \colon \Gamma_F \to \mr{SO}_{N+1}(\ov{\Z}_{\ell})$ of $\br_{w_i}$.
\end{proof}
\begin{rem}
  As in Remark \ref{notodd}, perhaps the most interesting cases here
  are when $N \equiv 2 \pmod 8$ and Case $6_{\Omega}$ when
  $N \equiv 6 \pmod 8$, since then we begin with $\mr{SO}_N$-valued
  representations that are not odd. Again, existing lifting techniques
  cannot lift these $\mr{SO}_N$-representations to Hodge--Tate regular
  geometric $\mr{SO}_N$-representations; of course, replacing
  $E[\ell]$ by $T_{\ell}(E)$ provides some geometric lift, but it has
  only three distinct Hodge--Tate weights, with high multiplicities
   (since this gives a submodule of
    $H^2(E, \mathbb{Z}_{\ell})$).
\end{rem}

\appendix

\section{The action of complex conjugation on the cohomology
  of a real elliptic surface}\label{complexconj}

 Recall that for any smooth variety $X/\R$ there is a ``transport of structure" isomorphism $\mr{F}_{\infty} \colon H^*(X(\C), \Q) \to H^*(X(\C), \Q)$ induced by the action of complex conjugation on the points of the manifold $X(\C)$; and that functoriality of the Betti-\'{e}tale comparison isomorphism
\[
H^*(X(\C), \Q) \otimes_{\Q} \Q_{\ell} \xrightarrow{\sim} H^*_{\acute{e}t}(X_{\C}, \Q_{\ell}),
\]
implies that the automorphism $\mr{F}_{\infty}$ corresponds to the action of complex conjugation $c \in \gal(\C/\R)$ on the \'{e}tale cohomology. Moreover, in the Hodge decomposition $H^r(X(\C), \Q) \otimes_{\Q} \C = \bigoplus_{p+q=r} H^{p, q}(X(\C))$, $\mr{F}_{\infty}$ exchanges $H^{p, q}(X(\C))$ and $H^{q, p}(X(\C))$. It follows that when $r$ is odd, $\tr(c|H^r_{\acute{e}t}(X_{\C}, \Q_{\ell}))=0$. For $r$ even, however, the contribution of $H^{\frac{r}{2}, \frac{r}{2}}(X(\C))$ terms can make it a subtle matter to compute the trace. In this section, we address this problem for the middle cohomology of certain elliptic surfaces over $\R$. We begin with a general lemma:\footnote{We thank Prakash Belkale for help in simplifying our original proof of this Lemma.}
\begin{lem} \label{lem:trace}%
  Let $M$ be a compact differentiable manifold and $f:M \to M$ a diffeomorphism such
  that $f^n = Id_M$ for some positive integer $n$. Let $F$ be the
  fixed point locus of $f$, i.e., the set of points $x \in M$ such
  that $f(x) = x$. Then
  \[
    \tr(f^*| H^*(M,\Q)) =  \chi(F)
  \]
  where the LHS is the alternating sum of the traces and $\chi$
  denotes the topological Euler characteristic, i.e., the trace of the
  identity map.
\end{lem}

\begin{proof}
  By averaging any Riemannian metric on $M$ with respect to the
  subgroup of $\mr{Diff}(M)$ generated by $f$, we see that $M$ has an
  $f$-invariant Riemannian metric $g$. Using the exponential map with
  respect to this metric at any point of $F$, we see that $F$ is a
  closed submanifold of $M$ (with connected components possibly of
  varying dimension). For $0 < \epsilon' < \epsilon \ll 1$, let
  $T_{\epsilon}$ (resp.~$T_{\epsilon'}$) be the tubular neighbourhood
  of $F$ in $M$ of radius $\epsilon$ (resp.~$\epsilon'$) constructed
  as in \cite[Theorem 11.1]{milnor-stasheff} using
  the metric $g$. Since $g$ is $f$-invariant, so are these tubular
  neighbourhoods.

  Let $A_{\epsilon}$ be the closure of $T_{\epsilon}$ in $M$ and
  $B_{\epsilon'} = M \bs T_{\epsilon'}$. Both of these sets are
  $f$-invariant compact manifolds with boundary. Since
  $B_{\epsilon'} \cap F = \emptyset $, $f$ has no fixed points on
  $B_{\epsilon'}$. Thus, by the ``no fixed points'' version of the
  Lefschetz fixed point theorem,
  $\tr(f^*| H^*(B_{\epsilon'}, \Q)) = 0$ and
  $\tr(f^*| H^*(A_{\epsilon} \cap B_{\epsilon'}, \Q)) = 0$.

  We now consider the
  Mayer--Vietoris sequence for the cover of $M$ given by
  $A_{\epsilon}$ and $B_{\epsilon'}$. Since $f$ preserves these sets,
  it induces an endomorphism of this sequence, and exactness implies
  that the alternating sums of the traces of $f^*$ on the terms of
  this sequence must be $0$.  Since the
  inclusion of $F$ in $A_{\epsilon}$ is a homotopy equivalence, we get
  that
  \[
    \tr(f^*| H^*(M,\Q)) = \tr(f^*| H^*(A_{\epsilon},\Q)) = \chi(F)
  \]
  as claimed.

  \end{proof}

  \begin{lem} \label{lem:conj}%
    Let $X$ be a smooth projective variety over $\R$. Then
  \[
    \tr(\mr{F}_{\infty}| H^*(X(\C), \Q)) = \chi(X(\R))
  \]
\end{lem}
\begin{proof}
  This follows immediately from Lemma \ref{lem:trace} since the fixed
  point locus of $\mr{F}_{\infty}$ is precisely $X(\R)$.
\end{proof}

We now turn to the study of real elliptic surfaces. The following lemma is standard for elliptic surfaces over $\C$. Let $C/\mathbb{R}$ be a smooth projective geometrically connected curve.
\begin{lem} \label{lem:chiell}%
  Let $\pi:E \to C$ be an elliptic surface over $\R$. Then
  \[
    \chi(E(\R)) = \sum_{x \in S} \chi(E_x(\R)) ,
  \]
  where $S \subset C(\R)$ is the set of real points over which $\pi$
  is not smooth and $E_x$ is the fibre of $\pi$ over $x$.
\end{lem}
\begin{proof}
  For a topological space $T$ we denote by $\chi_c(T)$ the Euler
  characteristic with compact supports. If $U\subset T$ is open and
  $Z = T \bs U$, then the long exact sequence of cohomology with
  compact supports implies that $\chi_c(T) = \chi_c(U) + \chi_c(Z)$ (as
  long as all the cohomology groups are finite dimensional).

  Let $U' = C(\R) \bs S$. The map $\pi_{\R}$ induced by $\pi$ from
  $E(\R) \to C(\R)$ is a proper fibre bundle over each connected
  component of $U'$ with fibre homeomorphic to a circle, two disjoint
  circles, or the empty set. The multiplicativity of Euler
  characteristics (with compact supports) for fibre bundles applied to
  the connected components of $U'$ then shows that
  $\chi_c(\pi^{-1}(U')) = 0$. The lemma follows from this, the
  additivity of $\chi_c$, and the fact that $E(\R)$ is compact.
\end{proof}

The following lemma gives the Euler characteristic of the real points
of the types of singular fibres (in the Kodaira classification) of
elliptic fibrations with a section over $\R$ that we need. We note
that in general this depends on the real structure of the fibre, not
just the Kodaira symbol.
\begin{lem} \label{lem:fibre}%
In the notation of the Kodaira classification (see, e.g., \cite[IV.8]{silverman:aec2}), and allowing all $n \geq 0$ in the $I_n^*$ case, we have: 
\begin{align*}
    \chi(I_1) &=
      \begin{cases}
        -1  & \text{ if the fibre is split, } \\
        1   & \text{ if the fibre is non-split;}
      \end{cases} \\
    \chi(I_2) & =
      \begin{cases}
        -2  & \text{ if the fibre is split, } \\
        0   & \text{ if the fibre is non-split;}
      \end{cases}\\
  \chi(II) &= 0 \\
  \chi(III) &= -1  \\
  %\chi(I_0^*) & =
     %           \begin{cases}
        %          -4 & \text{ if all components are defined over
           %         $\R$,} \\
             %     -2 & \text{ if three components are defined over
             %       $\R$.}
           %       \end{cases} \\
   \chi(I_n^*) &= 
   		\begin{cases}
                  -n-4 & \text{ if all components are defined over
                    $\R$,} \\
                  -n-2 & \text{ if all but two components are defined over
                    $\R$;}
                  \end{cases} \\
   \chi(III^*) &= -7.
   		%\begin{cases}
                  %-7 & \text{ if all components are defined over
                   % $\R$,} \\
                  %-3 & \text{ if all but two components are defined over
                   % $\R$,} \\
                    %-1 & \text{if all but four components are defined over $\R$.}
                  %\end{cases} 
\end{align*}
\end{lem}

\begin{proof}
  The lemma follows from \cite[Theorem 4.1]{silhol-RAS}, where
  the Euler characteristics of all possible fibres of real elliptic
  surfaces with a section are computed.
\end{proof}

Let $W \subset H^2(E(\C),\Q)$ be the subspace spanned
by the fundamental classes of all irreducible components of all the
(complex) fibres of $\pi$ as well as the fundamental class of a
section $C$ of $\pi$ (defined over $\R$). 
\begin{equation} \label{eq:N}
  W = \Q[C] \oplus \Q[E_{sm}] \oplus \bigoplus_{x \in C(\C)} \frac{(\oplus_i \Q[E_{x,i}])}{\Q[E_x]} ,
\end{equation}
where $E_x$ is the scheme theoretic fibre over $x$, the $E_{x,i}$ are
the reduced irreducible components of $E_x$, $E_{sm}$ is any smooth fibre
and $[D ]$, for an algebraic $1$-cycle $D$, denotes the fundamental
class in $H^2(E(\C), \Q)$. The summand corresponding to any
irreducible fibre $E_x$ is zero, so the sum is actually finite.

If a complex fibre is not defined over $\R$, complex conjugation maps
it to a distinct fibre, so the trace of $\mr{F}_{\infty}$ corresponding to
the subspace of $W$ spanned by all the irreducible components of all
these fibres is $0$. Thus, as far as computing the trace goes, we only
need consider $E_x$ for $x \in C(\R)$. In this case, the
corresponding summand $\frac{(\oplus_i \Q[E_{x,i}])}{\Q[E_x]}$ is
preserved by $\mr{F}_{\infty}$ so it suffices to consider the trace on each
such fibre separately. We now make a list of the possibilities in
terms of the Kodaira type and the real structure as in Lemma
\ref{lem:fibre}. 
\begin{lem} \label{lem:trfibre}%
For  a Kodaira symbol $*$ we denote
by $\tr(*)$, the trace of $\mr{F}_{\infty}$ acting on $\frac{\oplus_i
\Q[E_{x,i}])}{\Q[E_x]}$, where $E_x$ is a real fibre of type $*$. We have:
\begin{align*}
  \tr(I_1) & = 0\\
  \tr(I_2) & = -1 \\
  \tr(II) &  = 0 \\
  \tr(III) & = -1  \\
 % \tr(I_0^*) & =
    %           \begin{cases}
     %            -4 & \text{ if all components are defined over
      %             $\R$,} \\
       %          -2 & \text{ if three components are defined over
       %            $\R$.}
       %        \end{cases} \\
   \tr(I_n^*) & =
               \begin{cases}
                 -n-4 & \text{ if all components are defined over
                   $\R$,} \\
                 -n-2 & \text{ if all but two components are defined over
                   $\R$;}
               \end{cases}    \\
    \tr(III^*) & =
    		%\begin{cases}
                  -7. %& \text{ if all components are defined over
                    %$\R$,} \\
                  %-5 & \text{ if all but two components are defined over
                   % $\R$,} \\
                   % -3 & \text{if all but four components are defined over $\R$.}
                 % \end{cases}         
\end{align*}
\end{lem}

\begin{proof}
  The formulae follow from the description of the possible real
  structures in \cite[Theorem 4.1]{silhol-RAS} and the following
  elementary facts:
  \begin{itemize}
  \item If $E_{x,i}$ is a $\C$-irreducible component of $E_x$ defined
    over $\R$ then $\mr{F}_{\infty}$ acts on $\Q[E_{x,i}]$  by $-1$ since
    complex conjugation reverses orientation.
  \item If $E_{x,i}$ and $E_{x,j}$ are two distinct $\C$-irreducible
    components of $E_x$ which are conjugate, then the trace of
    $\mr{F}_{\infty}$ on $\Q [E_{x,i}] \oplus \Q[E_{x,j}]$ is $0$.
  \item $\mr{F}_{\infty}$ acts on $\Q[E_x]$ by $-1$.
  \end{itemize}
  
\end{proof}

The fact that the numbers associated to $I_n^*$ in Lemmas
\ref{lem:fibre} and \ref{lem:trfibre} are the same greatly simplifies
the computations to follow.

\subsection{}
Let $U$ be the maximal open subset of $C_{\C}$ over which the map
$\pi:E_{\C} \to C_{\C}$ is smooth and let $j:U(\C) \to
C(\C)$ be the inclusion. Let $\mc{E}^i$ be the local system on
$U(\C)$ given by $R^i \pi_* (\Q_{\pi^{-1}(U)})$ and let $\mc{F}$ be
the constructible sheaf $j_*(\mc{E}^1)$. Let $V = H^1(C(\C),
\mc{F})$. By the decomposition theorem (\cite{bbd}) and the description of intermediate extension on a smooth curve, we see (noting that the fibres of $\pi$ are connected) that
\begin{align}
R \pi_*(\Q[2]) &\cong j_*\mc{E}^0[2] \oplus j_* \mc{E}^1[1] \oplus j_* \mc{E}^2[0] \oplus \mc{P}_{C \setminus U} \\ 
& \cong \Q[2] \oplus j_* \mc{E}^1[1] \oplus \Q[0](-1) \oplus \mc{P}_{C \setminus U},
\end{align}
%The fibres of $\pi$ are connected, so we can apply \cite[Example 1.8.4]{decataldo-migliorini:decomposition} (see too \cite[Theorem 3.2.3]{decataldo-migliorini:intersection}) to find a direct sum decomposition
where $\mc{P}_{C\setminus U}$ is a sheaf supported on $C \setminus U$: it is clear from the decomposition theorem that $\mc{P}_{C \setminus U}$ must have punctual support, and to see that it is indeed a sheaf placed in degree zero we note that it must be Verdier self-dual. Its stalks are easily computed using the proper base-change theorem, and we obtain
\[
H^2(E(\C), \Q)= H^2(C(\C), \Q) \oplus H^1(C(\C), j_* \mc{E}^1) \oplus H^0(C(\C), \Q(-1)) \oplus \bigoplus_{x \in (C \setminus U)(\C)} 
\frac{(\oplus_i \Q[E_{x,i}])}{\Q[E_x]},
\]
 i.e. a direct sum decomposition 
\[
  H^2(E(\C), \Q) \cong V \oplus W.
\]
(Properties of the cycle class map imply that the cycle classes of the identity section and the smooth fibre, respectively, account for the terms corresponding to $\mc{E}^0$ and $\mc{E}^2$.)
Furthermore, since $U$ is defined over $\R$, complex
conjugation induces an involution on $V$ (which we also denote by
$\mr{F}_{\infty}$) and the isomorphism above is equivariant for this action.

%We now assume that $H^1(E(\C), \Q) = 0$.
Since $\tr(\mr{F}_{\infty}|H^r(E(\C), \Q))=0$ when $r$ is odd (in fact, for our elliptic surface over $\P^1$ with non-constant $j$-invariant these odd cohomology groups are zero), it is clear that using
Lemma \ref{lem:conj} we may compute the trace of $\mr{F}_{\infty}$ on $V$
if we know $\chi(E(\R))$ and the trace of $\mr{F}_{\infty}$ on $W$, since
the traces on both $H^0(E(\C),\Q)$ and $H^4(E(\C),\Q)$ are equal to
$1$.

\begin{prop} \label{prop:z}%
  Let $\pi:E \to C$ be an elliptic fibration over $\R$ with a section,
  and assume the singular fibres of $\pi$, defined over $\R$, are of
  the following forms:
  \begin{enumerate}
  \item
    There is one singular fibre of type $I_1$ that is not split over
    $\R$, a singular fibre of type $I_2$ that is not split over $\R$,
    a singular fibre of type $III$, and all other singular fibres
    defined over $\R$ are of type $I_0^*$.
  \item
    There is one singular fibre of type $I_1$ that is not split over
    $\R$, a singular fibre of type $I_1$ that is split over $\R$, two
    singular fibres of type $II$, and all other singular fibres defined
    over $\R$ are of type $I_0^*$.
   \item[(3${}_{\mr{O}}$)] There is one singular fibre of type $I_2$ that is split over $\R$, one singular fibre of type $I_2$ that is not split over $\R$, two singular fibres of type $I_4^*$, and all other singular fibres defined over $\R$ are of type $I_0^*$.
   \item[(3${}_{\Omega}$)] There are two singular fibres of type $I_2$ that are split over $\R$, two singular fibres of type $I_4^*$, and all other singular fibres defined over $\R$ are of type $I_0^*$.
   \item[(4)] There is one singular fibre of type $I_1$ that is not split over $\R$, one singular fibre of type $I_2$ that is split over $\R$, one singular fibre of type $III^*$, and all other singular fibres defined over $\R$ are of type $I_0^*$.
  \end{enumerate}
  Then 
  \[
  \tr(\mr{F}_{\infty}| V)= 
     \begin{cases} 2 & \text{in Case (1);} \\
     			  0 & \text{in Case (2);} \\
 			  0 & \text{in Case (3${}_{\mr{O}}$);} \\
			  -2 & \text{in Case (3${}_{\Omega}$);} \\
			  0 & \text{in Case (4).}% if all components of the $III^*$ fibre are defined over $\R$;} \\
			 % 2 & \text{in Case (4) if not all components of the $III^*$ fibre are defined over $\R$.}
     \end{cases}
     \]
\end{prop}

\begin{proof}
  By Lemma \ref{lem:conj} and the preceding discussion,
  \begin{equation} \label{eq:trV}
    \tr(\mr{F}_{\infty} | V ) = \chi(E(\R)) - (2 + \tr(\mr{F}_{\infty}|W)) .
  \end{equation}
  We now compute the RHS in each case.

  Using Lemma \ref{lem:chiell} and the list of singular fibres, we see
  that
  in Case (1)
  \[
    \chi(E(\R)) = 1 + 0 -1 -4a_1 -2a_2,
  \]
  where $a_1$ (resp.~$a_2$) is the number of fibres of type $I_0^*$ of
  the first (resp.~second) type. On the other hand, using Lemma
  \ref{lem:trfibre} and the list of singular fibres
  \[
    \tr(\mr{F}_{\infty}|W) = -2 + 0 -1 -1 -4a_1 -2a_2 ,
  \]
  where the first $-2$ corresponds to the sum of the traces on a
  smooth fibre and the section.  Inserting these numbers in
  \eqref{eq:trV} we get that the LHS is $2$ as claimed.

  By the same method, in Case (2) we have
   \[
    \chi(E(\R)) = 1 -1 +0  -4a_1 -2a_2,
  \]
  and
  \[
    \tr(\mr{F}_{\infty}|W) = -2 + 0 +0 +0 -4a_1 -2a_2 ,
  \]
  so in this case $\tr(\mr{F}_{\infty}|V) =0$ as claimed.
  
  In Case ($3_{\mr{O}}$), we again have cancellation of the contribution from $I_0^*$ and $I_n^*$ fibres, and we find
  \[
  \tr(\mr{F}_{\infty}|V)= -2+0 - (2 -1 -1 -1 -1)=0.
  \]
  
  In Case ($3_{\Omega}$), computing similarly we find $\tr(\mr{F}_{\infty}|V)=-2$.
  
  In Case (4), we likewise find $\tr(\mr{F}_{\infty}|V)= 0$.
 \iffalse \[
  \tr(\mr{F}_{\infty}|V)= 
  \begin{cases}
  	%-2+1-7-(2-1-1+0-1 -7)=
	0 & \text{if all components of the $III^*$ fibre are defined over $\R$;}\\
	%-2+1-3-(2-1-1+0-1 -5)=
	2 & \text{if all but two components of the $III^*$ fibre are defined over $\R$;}\\
	%-2+1-1-(2-1-1+0-1 -3)=
	2 & \text{if all but four components of the $III^*$ fibre are defined over $\R$.}\\
  \end{cases}
  \]
  \fi
\end{proof}

\def\cprime{$'$}
\providecommand{\bysame}{\leavevmode\hbox to3em{\hrulefill}\thinspace}
\providecommand{\MR}{\relax\ifhmode\unskip\space\fi MR }
\providecommand{\MRhref}[2]{%
  \href{http://www.ams.org/mathscinet-getitem?mr=#1}{#2}
}
\providecommand{\href}[2]{#2}

%\bibliographystyle{amsalpha}
%\bibliography{biblio.bib}

\end{document}